\documentclass[11pt,a4paper]{article}
\usepackage{amssymb,amsmath,amsthm}

\usepackage{times}
\def\de{{\rm d}}

\newtheorem{theorem}{Theorem}[section]

\newtheorem{corollary}{Corollary}[section]
\newtheorem{remark}{Remark}[section]
\numberwithin{equation}{section}

\title{Least squares volatility change point estimation for partially observed diffusion processes }
\begin{document}

\author{Alessandro De Gregorio,  \,\,Stefano M. Iacus\footnote{\textbf{email:} alessandro.degegorio@unimi.it, stefano.iacus@unimi.it}\\
{\it Dipartimento di Scienze Economiche, Aziendali e Statistiche}\\
{\it	Via Conservatorio 7,  20122 Milan - Italy}
}

\maketitle

\begin{abstract}
A one dimensional diffusion process $X=\{X_t, 0\leq  t \leq T\}$, with drift $b(x)$ and diffusion coefficient $\sigma(\theta, x)=\sqrt{\theta} \sigma(x)$ known up to $\theta>0$, is
supposed to switch volatility regime at some point $t^*\in (0,T)$. On the
basis of discrete time observations from $X$, the problem is the
one of estimating the instant of change in the volatility
structure  $t^*$ as well as the two values of $\theta$, say
$\theta_1$ and $\theta_2$, before and after the change point. It
is assumed that the sampling occurs at regularly spaced times
intervals of length $\Delta_n$ with $n\Delta_n=T$. To work out our statistical problem we use a least squares approach. Consistency, rates of
convergence and distributional results of the estimators are
presented under an high frequency scheme. We also study the case of a diffusion process with unknown drift and unknown volatility but constant. 
\end{abstract}


\noindent

\newpage

\section{Introduction}
Change point problems have originally arisen in the context of quality control, but the problem of abrupt changes in general arises in many contexts like epidemiology, rhythm analysis in electrocardiograms, seismic signal processing, study of archeological sites, financial markets. 

Originally, the problem was considered for i.i.d samples; see Hinkley (1971), Cs\"{o}rg\H{o} and Horv\'{a}th (1997), Inclan and Tiao (1994) and it moved naturally into the time series context as 
economic time series often exhibit prominent evidence for structural change in the underlying 
model; see, for example, Kim {\it et al.} (2000), Lee {\it et al.} (2003), Chen {\it et al.} (2005) and the papers cited therein. 

In this paper we deal with a change-point problem for the volatility of a diffusion process observed at discrete times. The instant of the change in volatility regime is identified retrospectively by the method of the least squares along the lines proposed in Bai (1994). For continuous time observations of diffusion processes Lee {\it et al.} (2006) considered the change point estimation problem for the drift.  In the present work the drift coefficient of the stochastic differential equation is assumed known and if not it is estimated nonparametrically. 

The paper is organized as follows. Section \ref{sec:model} introduces the model of observation and the estimator of the change point instant and the estimators of the volatilities before and after the change. Section \ref{consistency} analyzes the asymptotic properties of the estimators. In Section \ref{sec:nonp} considers the case when the drift is unknown and the diffusion coefficient does not depend on the state of the process. All the proofs of the Theorems are contained in the Appendix.

\section{The least squares estimator}\label{sec:model}
We denote by $X = \{X_t, 0 \leq t\leq T\}$ the diffusion process, with state space $I=(l,r),\, -\infty\leqslant l\leqslant r \leqslant +\infty$, solution of the stochastic differential equation
\begin{equation}
\de X_t = b(X_t) \de t + \sqrt{\theta} \sigma(X_t) \de W_t, 
\label{eq:sde}
\end{equation}
with $X_0=x_0$ and  $\{W_t, t\geq 0\}$ a standard Brownian motion.
 We suppose that the value of $\theta$ is $\theta_{1}$ up to some unknown time $t^*\in(0,T)$ and $\theta_{2}$ after, i.e. $\theta=\theta_1 1_{\{t\leqslant t^*\}}+\theta_2 1_{\{t>t^*\}}$. The parameters $\theta_1$ and $\theta_2$ belong to $\Theta$, a compact set of $\mathbb{R}^+$. The coefficients $b:I\to \mathbb{R}$ and $\sigma:I\to (0,\infty)$ are supposed to be known with continuous derivatives.  
The continuity of the derivatives of $b(\cdot)$ and $\sigma(\cdot)$ assures that it exists a unique continuous process solution to \eqref{eq:sde}, which is defined up an explosion time (see Arnold, 1974). Let $s(x)=\exp\left\{-\int_{x_0}^x 2b(u)/\sigma^2(u)du\right\}$ be the scale function (where $ x_0$ is an arbitrary point inside $I$). 
 \begin{itemize}
 \item[A1.] $\lim_{x_1\to l}\int_{x_1}^x s(u)du=-\infty$, 
 $\lim_{x_2\to r}\int_{x}^{x_2} s(u)du=+\infty$, where $l<x_1<x<x_2<r$.
 \end{itemize}
Condition A1 guarantees  that the exit time from $I$ is infinite (see Karatzas and Shevre, 1991).

The process $X$ is observed at $n+1$ equidistant discrete times $0=t_0<t_1<...<t_n=T,$ with $t_i =
i \Delta_n$, $n \Delta_n = T$. For the sake of simplicity we will write  $X_i = X_{t_i}$ and $W_{t_i}=W_i$. The asymptotic framework
is an high frequency scheme: $n \to \infty$, $\Delta_n\to 0$ with $n\Delta_n=T$.  Given the observations $X_i,i=0,1,...,n,$ the aim of this work is to estimate the change time $t^*$ as  well as the two quantities $\theta_1,\theta_2$.

In order to obtain a simple least squares estimator, we follow the same approach proposed in Bai (1994). To this end we make use of Euler approximation to the solution of \eqref{eq:sde}, i.e.
$$
X_{i+1} = X_i + b(X_i) \Delta_n + \sqrt{\theta} \sigma(X_i)
(W_{i+1}-W_i),$$
and introduce the quantities
$$
Z_i = \frac{X_{i+1}-X_{i} - b(X_i)\Delta_n}{\sqrt{\Delta_n} \sigma(X_i)} =
\sqrt{\theta} \frac{W_{i+1}-W_{i}}{\sqrt{\Delta_n}},\quad i=1,...,n,
$$
which represent $n$ independent standard normal variables. 

We denote by $k_0 = [n\tau_0]$ and $k = [n\tau]$, $\tau, \tau_0 \in (0,1)$, where $[x]$ is the integer part of the real value $x$. Given that $\Delta_n\to 0$, without loss of generality, we can assume that the process switches volatility regime exactly at time $t_i=t_{k_0}=k_0\Delta_n=t^*$.
The least squares estimator of the change point is obtained as follows
\begin{eqnarray}
\hat k_0& =& \arg\min_k\left( \min_{\theta_1, \theta_2}
\left\{\sum_{i=1}^{k} (Z_i^2 - \theta_1)^2 + \sum_{i=k+1}^n (Z_i^2 -
\theta_2)^2 \right\}
\right)\notag\\
&=&\arg\min_k\left\{\sum_{i=1}^{k} (Z_i^2 - \bar \theta_1)^2+\sum_{i=k+1}^n (Z_i^2 -\bar \theta_2)^2\right\},
\label{eq:ls}
\end{eqnarray}
where
 \begin{eqnarray*}
\bar \theta_1=\arg \min_{\theta_1}\sum_{i=1}^{k} (Z_i^2 - \theta_1)^2=\frac1k\sum_{i=1}^k Z_i^2=\frac{S_k}{k}
\end{eqnarray*}
 \begin{eqnarray*}
\bar \theta_2=\arg \min_{\theta_2}\sum_{i=k+1}^{n} (Z_i^2 - \theta_2)^2=\frac{1}{n-k}\sum_{i=k+
1}^n Z_i^2=\frac{S_{n-k}}{k}
\end{eqnarray*}
and $k=1,...,n-1.$ We introduce the following quantity
$$
U_k^2 = \sum_{i=1}^{k}(Z_i^2 - \bar\theta_{1})^2 +
 \sum_{i=k+1}^n (Z_i^2 - \bar\theta_{2})^2
$$
then we have that
$$
\hat k_0 = \arg\min_k U_k^2
$$

To study the asymptotic properties of $U_k^2$ it is better to
rewrite it in the following way
$$
U_k^2 = \sum_{i=1}^n (Z_i^2 - \bar Z_n)^2 - n V_k^2
$$
where $\bar Z_n = \frac{1}{n}\sum_{i=1}^n Z_i^2=\frac{1}{n}S_n$
and

$$
V_k= \left(\frac{k(n-k)}{n^2}\right)^\frac12\left(\bar\theta_{2}-\bar\theta_{1}\right)=\frac{S_n D_k}{\sqrt{k(n-k)}}
$$
with
$
 D_k =k/n - S_k/S_n.
 $

This representation of $U_k^2$ is obtained by lengthy but straightforward algebra and it is rather useful because minimization of $U_k^2$ is equivalent to the maximization of $V_k$ and hence $D_k$. So it is easier to consider the following  estimator of $k_0$
\begin{equation}\label{est}
\hat k_0 = \arg\max_k |D_k| = \arg\max_k(k(n-k))^\frac12|V_k|
\end{equation}

As a side remark, it can be noticed that for fixed $k$ (and under suitable hypothesis), $D_k$ can be seen as an approximate likelihood ratio statistics for testing the null hypothesis of no change in volatility (see e.g. Inclan and Tiao, 1994). We do not discuss approximate likelihood approach in this paper.

Once $\hat k_0$ has been obtained, the following estimator of the parameters $\theta_1$ and $\theta_2$ can be used
\begin{equation}\label{est1}
\hat\theta_{1} = \frac{S_{\hat k_0}}{\hat k_0},\quad
\hat\theta_{2} = \frac{S_{n-\hat k_0}}{n-\hat k_0 }.
\end{equation}
We will prove consistency of $\hat k_0$, $\hat\theta_{1}$ and $\hat\theta_{2}$ and also distributional results for these estimators.

Our first result concerns the asymptotic distribution of the statistic $D_k$ under the condition that no change of volatility occurs during the interval $[0,T]$.

\begin{theorem}\label{H0}
Assume  that $H_0$: $\theta_1=\theta_2=1$ , then we have that
\begin{equation}
 \sqrt{\frac{n }{2}} |D_k|\stackrel{d}{\to}
 |W^0(\tau)|
\end{equation}
where $\{W^0(\tau),0\leqslant \tau \leqslant 1\}$ is a Brownian bridge.
\end{theorem}

\begin{corollary}
From Theorem \ref{H0} we derive immediately that for $\delta \in (0,1/2)$
\begin{equation}
\sqrt{\frac{n}{2}}\sup_{\delta n\leq k \leq (1-\delta)n}|D_k|\stackrel{d}{\to}
\sup_{\delta\leq \tau \leq (1-\delta)}|W^0(\tau)|,
\end{equation}
\begin{equation}\label{ash0}
\sqrt{\frac{n}{2}}\sup_{\delta n\leq k \leq (1-\delta)n}|V_k|\stackrel{d}{\to}
\sup_{\delta\leq \tau \leq (1-\delta)}(\tau(1-\tau))^{-1/2}|W^0(\tau)|.
\end{equation}

\end{corollary}

The last asymptotic results are useful to test if a change point occurred in $[0,T]$. In particular it is possible to obtain the asymptotic critical values for the distribution \eqref{ash0}  by means of the same arguments used in  Cs\"{o}rg\H{o} and Horv\'{a}th (1997), pag. 25.

\section{Asymptotic properties of the estimator}\label{consistency}

We study the main asymptotic properties of the least squares estimator $\hat k_0$. We start analyzing the consistency and the rate of convergence of the change point estimator \eqref{est}. It is convenient to note that the rate of convergence is particularly important not only to describe how fast the estimator converges to the true value, but also to get the limiting distribution. The next Theorem represents our first result on the consistency.
\begin{theorem}\label{teo1}
The estimator $\hat\tau_0=\frac{\hat{k}_0}{n}$ satisfies
\begin{equation}\label{result1}
|\hat\tau_0-\tau_0|=n^{-1/2}(\theta_2-\theta_1)^{-1}O_p(\sqrt{\log n})
\end{equation}
\end{theorem}

Theorem \ref{teo1} implies consistency of our estimator, in fact we have that $n^\beta(\hat\tau_0-\tau_0)\to0$ in probability for any $\beta\in(0,1/2)$. We are able to improve the rate of convergence of $\hat\tau_0$.

\begin{theorem}\label{consim}
We have the following result
\begin{equation} \label{imrate}
\hat\tau_0-\tau_0=O_p\left(\frac{1}{n(\theta_2-\theta_1)^2}\right).
\end{equation}
\end{theorem}

It is also possible to derive the asymptotic distribution of $\hat\tau_0$ under our limiting framework for small variations of the rate of change of the direction. The case $\vartheta_n=\theta_2-\theta_1$ equal to a constant is less interesting because when $\vartheta_n$ is large the estimate of $k_0$ is quite precise. By adding the condition
\begin{itemize}
\item[A2.] Assume that $$\vartheta_n\to 0,\quad \frac{\sqrt{n}\vartheta_n}{\sqrt{\log n}}\to \infty$$
\end{itemize}
Under A2 the consistency of $\hat\tau_0$ follows immediately either from Theorem \ref{teo1} or Theorem \ref{consim}. In order to obtain the next result, it is useful to observe that
\begin{equation}
\hat k_0=\arg\max_kV_k^2=\arg\max_kn(V_k^2-V_{k_0}^2)
\end{equation}
and to define a two-sided Brownian motion $\mathcal{W}(u)$ in the following manner
\begin{equation}
\mathcal{W}(u)=
\begin{cases}
W_1(-u),& u<0 \\
W_2(u), &u\geq0
\end{cases}
\end{equation}
where $W_1,W_2$ are two independent Brownian motions. Now we present the following convergence in distribution result.
\begin{theorem}\label{dist}
Under assumption A2 we have that
\begin{equation}
\frac{n \vartheta_n^2(\hat\tau_0-\tau_0)}{2\tilde\theta^2}\stackrel{d}{\to}
\arg\max_v\left\{\mathcal{W}(v)-\frac{|v|}{2}\right\},
\end{equation}
where $W(v)$ is a two-sided Brownian motion and $\tilde\theta$ is a consistent estimator for $\theta_1$ or $\theta_2$.
\end{theorem}

Let $\theta_0$ be the limiting value of both $\theta_1$ and $\theta_2$. Using the consistency result, we are able to obtain the asymptotic distributions for the estimators $\hat \theta_1,\hat\theta_2$, defined in \eqref{est1}.
\begin{theorem}\label{distes}
Under assumption A2 we have that
\begin{equation}\label{estcon}
\sqrt{n}\left( \begin{array}{c}
\hat\theta_1-\theta_1 \\
\hat\theta_2-\theta_2 \\
\end{array} \right) \stackrel{d}{\to}N\left( 0,\Sigma \right),
\end{equation}
where
\begin{equation}
\Sigma=\left( \begin{array}{cc}
2\tau_0^{-1}\theta_0^2& 0\\
0 &2 (1-\tau_0)^{-1}\theta_0^2 \\
\end{array} \right).
\end{equation}
\end{theorem}
\begin{remark}
It is easy to verify that the Theorems presented in this Section are also true if we consider an horizon time tending to infinite, i.e. $\Delta_n\to 0$ and $n\Delta_n=T\to \infty$ as $n\to \infty$.
\end{remark}
\paragraph{Remark on ergodic case.}
If the Euler approximation is not admissible, it is worth to consider the ergodic case.
Let $m(u)=(\sigma^2(u)s(u))^{-1}$ be the speed measure of the diffusion process $X_t$. We introduce the following assumptions:
 \begin{itemize}
 \item[A3.] $\int_l^r m(x)dx< +\infty$.
 \item[A4.] $X_0=x_0$ has distribution $P^0$.
 \end{itemize}
Assumptions A1, A3 and A4 imply that the process $X_t$ is also ergodic and strictly stationary with invariant distribution $P^0(dx)=\pi(x)dx,\, \pi(x)=m(x)/ \int_l^r m(u)du$. Under the additional condition:
\begin{itemize}
\item[A5.]  $\lim_{x\to 0}\sigma(x) \pi(x)=0$ or $\lim_{x\to \infty}\sigma(x) \pi(x)=0$ and
$$\lim_{x\to 0} \left|\frac{\sigma(x)}{2b(x)-\sigma(x)\sigma'(x)}\right|<\infty\,\, or \,\, \lim_{x\to \infty} \left|\frac{\sigma(x)}{2b(x)-\sigma(x)\sigma'(x)}\right|<\infty$$
\end{itemize}
the observed data $X_i,i=0,1,...,n,$ is a strictly stationary $\beta$-mixing sequence satisfying $k^\delta \beta_k\to 0$ as $k\to\infty $ for some fixed $\delta>1$ (see e.g A\"it-Sahalia, 1996). 

Under this setup and $n\Delta_n=T\to \infty$, similar results to the ones presented in the above can be obtained using the same techniques of Chen {\it et al.} (2005).

\section{Estimation of the change point with unknown drift}\label{sec:nonp}

We want to analyze the change point problem for a diffusion process $X_t$, when the drift coefficient $b(X_t)$ is unknown, while the diffusion coefficient is supposed unknown but independent from the state of the process $X_t$. In other words the process $X_t$ is the solution of the following reduced stochastic differential equation
 
 \begin{equation}\label{model2}
 \de X_t=b(X_t)\de t+\sqrt{\theta} \de W_t,
 \end{equation}
and the observation scheme and the asymptotics are as in Section \ref{sec:model}. Let $K\geqslant 0$ be a kernel function, i.e. $K$ is symmetric and continuously differentiable, with $\int_{\mathbb{R}}uK(u)du=0$, $\int_{\mathbb{R}}K^2(u)du<\infty$ and such that
$\int _{\mathbb{R}}K(u)du=1.$ We start introducing the following quantities

$$\hat Z_i=\frac{X_{i+1}-X_i -\hat b(X_i)\Delta_n}{
\sqrt{\Delta_n}}$$
 where

\begin{equation}
\hat b(x)=\frac{\sum_{i=1}^nK\left(
\frac{X_i-x}{h_n}\right)\frac{X_{i+1}-X_i 
}{\Delta_n}}{\sum_{i=1}^nK\left( \frac{X_i-x}{h_n}\right)}
\end{equation}
is a nonparametric estimator of the drift constructed using the full sample and $h_n$ is the bandwidth defined as in Silverman (1986).  The least squares estimator takes the following form  
 \begin{equation}\label{cpnonp}
\tilde k_0=\arg \min_k\left\{\sum_{i=1}^k \left(\hat Z_i^2-\bar {\theta}_1^*\right)^2+\sum_{i=k+1}^n \left(\hat Z_i^2-\bar {\theta}_2^*\right)^2\right\},
\end{equation}
where 
$$\bar {\theta}_1^*=\frac{\hat S_k}{k}, \quad \bar {\theta}_2^*=\frac{\hat S_{n-k}}{n-k}\,,$$
with
$\hat{S}_k=\sum_{i=1}^k \hat Z_i^2,  \hat{S}_{n-k}=\sum_{i=k+1}^n\hat  Z_i^2$. From \eqref{cpnonp}, by the same steps considered in the Section \ref{sec:model}, we derive
 \begin{equation}
 \hat V_k= \left(\frac{k(n-k)}{n^2}\right)^\frac12\left(  \bar {\theta}_2^*- \bar {\theta}_1^*\right)
 \end{equation}

We are able to show that the asymptotic properties of the estimator $\tilde k_0$ defined in \eqref{cpnonp}, are equal to the ones of $\hat k_0$.

\begin{theorem}\label{teonon}
The same results presented in the Theorems \ref{teo1}-\ref{distes} hold for the estimator $\tilde k_0$.
\end{theorem}

\section{Appendix}
As in Chen {\it et al.} (2005) some of the proofs are based on the ones in Bai (1994). We adapt Bai's theorems making the appropriate (but crucial) adjustments when needed, skipping all algebraic calculations which can be found in the original paper of the author.

\begin{proof}[Proof of Theorem \ref{H0}]
By setting $\xi_i=Z_i^2-1,$ under $H_0$ we note that $E (\xi_i)=0$ and $Var (\xi_i)=2.$ Let us introduce the quantity
$$
 Y_n(\tau)=\frac{1}{\sqrt{2n}} \mathcal{S}_{[n\tau]}+(n\tau-[n\tau]) \frac{1}{\sqrt{2n}}\xi_{[n\tau]+1},
$$
 where $\mathcal{S}_n=\sum_{i=1}^n\xi_i$.  It's no hard to see that
\begin{equation}\label{ugp}
\left|\frac{1}{\sqrt{2n}}\sum_{i=1}^{[n\tau]}\xi_i-\frac{\sqrt{\tau}}{\sqrt{2[n\tau]}}\sum_{i=1}^{[n\tau]}\xi_i\right|\stackrel{p}{\to}0,
\end{equation}
and
$Var\left(\frac{\sqrt{\tau}}{\sqrt{2[n\tau]}}\sum_{i=1}^{[n\tau]}\xi_i\right)
=\tau$. Since the Lindeberg condition is true
\begin{equation}\label{Lin}
\sum_{i=1}^{[n\tau]}\frac{E\left\{\mathbf{1}_{|\xi_i|\geq
\varepsilon
\sqrt{2n}\}}\xi_i^2\right\}}{2n}\to
0,
\end{equation}
 we can conclude that

\begin{equation}\label{inv}
\frac{1}{
\sqrt{2n}}\mathcal{S}_{\left[n\tau\right]}\stackrel{d}{\to}N(0,\tau).
\end{equation}

 By Donsker's theorem we have $Y_n\stackrel{d}{\to}W(\tau)$ that implies $Y_n(\tau)-\tau Y_n(1)\stackrel{d}{\to}W^0(\tau)$, where $W(\tau)$ and $W^0(\tau)$ are respectively a standard Brownian motion and a Brownian bridge. Let $[n\tau]=k,k=1,2,...,n,$ then
 \begin{eqnarray}\label{h0}
Y_n(\tau) -\tau Y_n(1)&=&\frac{1}{\sqrt{2n}}\mathcal{S}_{[n\tau]}-\frac{\tau}{\sqrt{2n}}\mathcal{S}_{n}+ (n\tau-[n\tau]) \frac{1}{\sqrt{2n}}\xi_{[n\tau]+1}\notag\\
&=&\frac{1}{\sqrt{2n}}\left[\mathcal{S}_k-\frac{k}{n}\mathcal{S}_n\right]+ (n\tau-[n\tau]) \frac{1}{\sqrt{2n}}\xi_{[n\tau]+1}.
 \end{eqnarray}
 Now, by observing that
 $$\mathcal{S}_k-\frac{k}{n}\mathcal{S}_n=\left[\sum_{i=1}^k(Z_i^2-1)-\frac{k}{n}\sum_{i=1}^n(Z_i^2-1)\right]=-D_k\sum_{i=1}^nZ_i^2,$$
from \eqref{h0} we have that
\begin{equation}
\sqrt{\frac{n}{2}}|D_k|\frac{\sum_{i=1}^nZ_i^2}{ n}=\left|X_n(\tau)-\tau X_n(1)-\frac{(n\tau-[n\tau]) }{\sqrt{2n}}\xi_{[n\tau]+1}\right|.
\end{equation}
As $n\to\infty,\Delta_n\to 0$ we get that $\frac{\sum_{i=1}^nZ_i^2}{ n}\to 1$
 and
$
\frac{(n\tau-[n\tau]) }{\sqrt{2n}}\xi_{[n\tau]+1}\stackrel{p}{\to} 0.
$
Hence the thesis of Theorem follows.

\end{proof}
\begin{proof}[Proof of Theorem \ref{teo1}]
By the same arguments of Bai (1994), Section 3 and by using the formulas (10)-(14) therein, we have that
\begin{equation}\label{boundin}
|\hat\tau_0-\tau_0|\leq  C_{\tau_0} (\theta_2-\theta_1)^{-1}\sup_k|V_k-EV_k|,
\end{equation}
where $C_{\tau_0}$
is a   constant depending only on $\tau_0$. Furthermore
\begin{eqnarray*}
V_k-EV_k&=&\frac{1}{\sqrt{n}}\sqrt{\frac{k}{n}}\frac{1}{\sqrt{n-k}}\sum_{i=k+1}^n ( Z_i^2-\theta_2)\\
&&+\frac{1}{\sqrt{n}}\sqrt{1-\frac{k}{n}}\frac{1}{\sqrt{k}}\sum_{i=1}^k (Z_i^2-\theta_1)
\end{eqnarray*}
then we can write
\begin{equation}\label{intp}
|V_k-EV_k|\leq \frac{1}{\sqrt{n}}\left\{\frac{1}{\sqrt{n-k}}\left|\sum_{i=k+1}^n ( Z_i^2-\theta_2)\right|+\frac{1}{\sqrt{k}}\left|\sum_{i=1}^k (Z_i^2-\theta_1)\right|\right\}.
\end{equation}
By applying Haj\'ek-Renyi inequality for martingales we have that
\begin{eqnarray}\label{eq:renyi}
P\left\{\max_{1\leq k \leq n}\left|\frac{\sum_{i=1}^k (Z_i^2-\theta_1)}{c_k}\right|>\alpha \right\} &\leq& \frac{1}{\alpha^2}\sum_{k=1}^n\frac{E(Z_k^2 -\theta_1) ^2}{c_k^2}\notag\\
&=& \frac{  2\theta_1^2}{\alpha^2} \sum_{k=1}^n\frac{1}{c_k^2}\end{eqnarray}
Choosing $c_k=\sqrt{k}$ and observing that $\sum_{k=1}^n k^{-1} \leq  C \log n$, for some $C>0$ (see e.g. Bai, 1994),  we have that
\begin{equation}\label{ratecon}
\max_{1\leq k \leq n}\frac{1}{\sqrt{k}}\sum_{i=1}^k Z_i=O_p\left(\sqrt{\log n}\right).
\end{equation}
The same result holds for $\frac{1}{\sqrt{n-k}}\left|\sum_{i=k+1}^n ( Z_i^2-\theta_2)\right|$, then from the relationships \eqref{intp} and \eqref{ratecon} we obtain the result \eqref{result1}.
\end{proof}
\begin{proof}[Proof of Theorem \ref{consim}]
We use the same framework of the proof of the Proposition 3 in Bai (1994), Section 4, therefore we omit the details.

We choose a $\delta>0$ such that $\tau_0\in (\delta,1-\delta)$. Since $\hat k_0/n$ is consistent for $\tau_0$, for every $\varepsilon>0$, $Pr\{\hat k_0/n\not\in (\delta,1-\delta)\}<\varepsilon$ when $n$ is large. In order to prove \eqref{imrate} it is sufficient to show that $Pr\{|\hat\tau_0-\tau_0|>M(n\vartheta_n^2)^{-1}\}$ is small when $n$ and $M$ are large, where $\vartheta_n=\theta_2-\theta_1$. We are interested to study the behavior of $V_k$ for $n\delta\leq k \leq n(1-\delta)$, $0<\delta<1$. We define for any $M>0$ the set $D_{n,M}=\{k:n\delta\leq k \leq n(1-\delta),|k-k_0|>M\vartheta_n^{-2}\}$. Then we have that
$$Pr\{|\hat\tau_0-\tau_0|>M(n\vartheta_n^2)^{-1}\}\leq \varepsilon+Pr\{\sup_{k\in D_{n,M}}|V_k|\geq |V_{k_0}|\},$$
for every $\varepsilon>0.$ Thus we study the behavior of $Pr\{\sup_{k\in D_{n,M}}|V_k|\geq |V_{k_0}|\}$. It is possible to prove that
\begin{equation}
\label{int1}
\begin{aligned}
Pr\left\{\sup_{k\in D_{n,M}}|V_k|\geq |V_{k_0}|\right\}\leq&\quad Pr\left\{\sup_{k\in D_{n,M}} V_k - V_{k_0}\geq 0\right\}\\
& +Pr\left\{\sup_{k\in D_{n,M}}V_k+V_{k_0}\leq 0\right\}\\
=&\quad P+Q
\end{aligned}
\end{equation}
Furthermore
\begin{eqnarray}\label{int2}
 Q
 &\leq& 2Pr\left\{\sup_{k\leq n(1-\delta)}\frac{1}{n-k}\left|\sum_{i=k+1}^n (Z_i^2-\theta_2)\right|\geq \frac{1}{4}EV_{k_0}\right\}\\
 &&+2Pr\left\{\sup_{k\geq n\delta}\frac{1}{k}\left|\sum_{i=1}^k (Z_i^2-\theta_1)\right|\geq \frac{1}{4}EV_{k_0}\right\}.\notag
\end{eqnarray}
By observing that $\sum_{i=m}^\infty i^{-2}=O(m^{-1})$,  the Haj\'{e}k-Renyi inequality yields
\begin{equation}\label{hr2}
P\left\{\max_{ k\geq  m}\left|\frac{1}{k}\sum_{i=1}^k (Z_i^2-\theta_1)\right|>\alpha \right\} \leq \frac{C_1}{\alpha^2m},
\end{equation}
for some constant $C_1<\infty$.
The inequality \eqref{hr2} implies that \eqref{int2} tends to zero as $n$ tends to infinity.

Let $d(k)=\sqrt{((k/n)(1-k/n))}, k=1,2,...,n,$ for the first term in the right-hand of \eqref{int1} we have that
\begin{equation}
\label{int3}
\begin{aligned}
P\leq & \quad Pr\left\{\sup_{k\in D_{n,M}}\frac{n}{|k_0-k|}|G(k)|>\frac{\vartheta_nC_{\tau_0}}{2}\right\}\\
&+Pr\left\{\sup_{k\in D_{n,M}}\frac{n}{|k_0-k|}|H(k)|>\frac{\vartheta_nC_{\tau_0}}{2}\right\}\\
=&\quad P_1+P_2,
\end{aligned}
\end{equation}
where
\begin{equation}\label{g}
G(k)=d(k_0)\frac{1}{k_0}\sum_{i=1}^{k_0}(Z_i^2-\theta_1)-d(k)\frac{1}{k}\sum_{i=1}^{k}(Z_i^2-\theta_1)
\end{equation}
\begin{equation}\label{h}
H(k)=d(k)\frac{1}{n-k}\sum_{i=k+1}^{n}(Z_i^2-\theta_2)-d(k_0)\frac{1}{n-k_0}\sum_{i=k_0+1}^{n}(Z_i^2-\theta_2)
\end{equation}
We prove that $P_1$ tends to zero when $n$ and $M$ are large. Thus
we consider only $k\leq k_0$ or more precisely those values of $k$ such that $n\delta\leq k \leq n\tau_0-M\vartheta_n^{-2}$. For $k\geq n \delta$, we have
\begin{equation}\label{modG}
|G(k)|\leq \frac{k_0-k}{n\delta k_0}\left|\sum_{i=1}^{k_0}(Z_i^2-\theta_1)\right|+B\frac{k_0-k}{n}\frac{1}{n\delta }\left|\sum_{i=1}^{k}(Z_i^2-\theta_1)\right|+\frac{1}{n\delta }\left|\sum_{i=k+1}^{k_0}(Z_i^2-\theta_1)\right|,
\end{equation}
where $B\geq 0$ satisfies $|d(k_0)-d(k)|\leq B|k_0-k|/n$.
By means of \eqref{eq:renyi}, \eqref{hr2} and \eqref{modG},  we obtain
\begin{eqnarray*}
P_1&\leq& Pr\left\{\frac{1}{n\tau_0}\left|\sum_{i=1}^{[n\tau_0]}(Z_i^2-\theta_1)\right|>\frac{\delta\vartheta_n C_{\tau_0}}{6}\right\}\\
&&+Pr\left\{\sup_{1\leq k \leq n}\frac{1}{n}\left|\sum_{i=1}^{k}(Z_i^2-\theta_1)\right|>\frac{\delta\vartheta_n C_{\tau_0}}{6B}\right\}\\
&&+Pr\left\{\sup_{ k \leq n\tau_0-M\vartheta_n^{-2}}\frac{1}{n\tau_0-k}\left|\sum_{i=k+1}^{[n\tau_0]}(Z_i^2-\theta_1)\right|>\frac{\delta\vartheta_n C_{\tau_0}}{6}\right\}\\
&\leq&\frac{36 \theta_1^2 }{(\delta C_{\tau_0})^2 \tau_0 n\vartheta_n^2}+\frac{36\theta_1^2B^2}{(\delta C_{\tau_0})^2  n\vartheta_n^2}+\frac{36\theta_1^2}{\delta C_{\tau_0}^2M}.
\end{eqnarray*}
 When $n$ and $M$ are large the last three terms are negligible. Analogously we derive the proof of $P_2$.

 \end{proof}
 \begin{proof}[Proof of Theorem \ref{dist}]
The proof follows the same steps in Bai (1994), Theorem 1, hence we only sketch the parts of the proof that differ.
We consider only $v\leq 0$ because of symmetry.
Let $K_n(v)=\{k:k=[k_0+v\vartheta_n^{-2}], -M\leq v \leq 0, M>0\}$ and
\begin{equation}
\Lambda_n(v)=n(V_k^2-V_{k_0}^2)
\end{equation}
with $k\in K_n(v).$ We note that
\begin{eqnarray}\label{con1}
n(V_k^2-V_{k_0}^2)&=&2nEV_{k_0}(V_k-V_{k_0})\notag\\
&&+2n(V_{k_0}-EV_{k_0})(V_k-V_{k_0})\notag\\
&&+n(V_k-V_{k_0})^2
\end{eqnarray}
The last two terms in \eqref{con1} are negligible on $K_n(v)$. Since $\sqrt{n}(V_{k_0}-EV_{k_0})$ is bounded by \eqref{intp}, we have to show that $\sqrt{n}|V_k-V_{k_0}|$ is bounded. In particular, we can write
$$\sqrt{n}|V_k-V_{k_0}|\leq \sqrt{n}|G(k)+H(k)|+\sqrt{n}|EV_k-EV_{k_0}|,$$
where $G(k)$ and $H(k)$ are defined respectively in \eqref{g} and \eqref{h}. The upper bound \eqref{modG} is $o_p(1)$, because the first term is such that
\begin{equation}\label{smallo}
\begin{aligned}
\sqrt{n}\frac{k_0-k}{n\delta k_0}\left|\sum_{i=1}^{k_0}(Z_i^2-\theta_1^2)\right|&
\leq \frac{M}{\delta\tau_0 n\vartheta_n^2}\left|\frac{1}{\sqrt{n}}\sum_{i=1}^{k_0}(Z_i^2-\theta_1)\right|\\
&=\frac{O_p(1)}{n\vartheta_n^2}=o_p(1),
\end{aligned}
\end{equation}
similarly for the second term and for the third term we apply the invariance principle \eqref{inv}.
Now we explicit the limiting distribution for
\begin{equation}
2nEV_{k_0}(V_k-V_{k_0})=2\sqrt{\tau_0(1-\tau_0)}n\vartheta_n(V_{[k_0+v\vartheta_n^{-2}]}-V_{k_0}).
\end{equation}
For simplicity we shall assume that $k_0+v\vartheta_n^{-2}$ and $v\vartheta_n^{-2}$ are integers. We observe that
\begin{equation}
n\vartheta_n(V_k-V_{k_0})=n\vartheta_n(G(k)+H(k))-n\vartheta_n(EV_{k_0}-EV_k),
\end{equation}
where $G(k),H(k)$ are defined in the expressions \eqref{g}, \eqref{h}. We can rewrite $G(k)$ as follws
\begin{equation}
\label{gbis}
\begin{aligned}
G(k) =&  \quad d(k_0)\frac{k-k_0}{kk_0}\sum_{i=1}^{k_0}(Z_i^2-\theta_1)+\frac{d(k_0)-d(k)}{k}\sum_{i=1}^k(Z_i^2-\theta_1)\\
&+d(k_0)\frac{1}{k}\sum_{i=k+1}^{k_0}(Z_i^2-\theta_1).
\end{aligned}
\end{equation}
By the same arguments used to prove \eqref{smallo} we can show that the first two terms in \eqref{gbis} multiplied by $n\vartheta_n$ are negligible on $K_n(M)$. Furthermore $d(k_0)=\sqrt{\tau_0(1-\tau_0)}$ and $n/k\to 1/\tau_0$ for $k\in K_n(M)$, then we get that
\begin{eqnarray}
n\vartheta_n G(k_0&+&v\vartheta_n^{-2})=n\vartheta_nd(k_0)\frac{1}{k}\sum_{i=k+1}^{k_0}(Z_i^2-\theta_1)+o_p(1)\notag\\
&=&d(k_0)\frac{n}{k}\left\{\vartheta_n\sum_{i=1}^{|v|\vartheta_n^{-2}}(Z_{i+k}^2-\theta_1)\right\}+o_p(1)\notag\\
&\stackrel{d}{\to}&\frac{\sqrt{(1-\tau_0)\tau_0}}{\tau_0} \sqrt{2}\theta_1W_1(-v)
\end{eqnarray}
where in the last step we have used the invariance principle \eqref{inv}. Analogously we can show that
\begin{equation}
n\vartheta_n H(k_0+v\vartheta_n^{-2})
\stackrel{d}{\to}\frac{\sqrt{(1-\tau_0)\tau_0}}{1-\tau_0} \sqrt{2}\theta_1W_1(-v).
\end{equation}
Since
\begin{equation}
n\vartheta_n(EV_{k_0}-EV_k)\to \frac{|v|}{1\sqrt{\tau_0(1-\tau_0)}}
\end{equation}
we obtain that
\begin{equation}
\Lambda_n(v)\stackrel{d}{\to}2\left\{ \sqrt{2}\theta_1W_1(-v)-\frac{|v|}{2}\right\}.
\end{equation}
In the same way, for $v>0$, we can prove that
\begin{equation}
\Lambda_n(v)\stackrel{d}{\to}2\left\{ \sqrt{2}\theta_1W_2(v)-\frac{|v|}{2}\right\}.
\end{equation}
By applying the continuous mapping theorem and Theorem \ref{consim}.
\begin{equation}
\frac{n\vartheta_n^2(\hat\tau_0-\tau_0)}{2\tilde \theta^2}\stackrel{d}{\to}
\frac{1}{2\theta_1^2}\arg\max_v\Lambda_n(v).
\end{equation}
Since $aW(v)\stackrel{d}{=}W(a^2 v),a\in \mathbb{R}$, a change in variable transforms
$\arg\max_v\Lambda_n(v)$ into
  $2\theta_1^2\arg \max_v\left\{W(v)-\frac{|v|}{2}\right\},$ which concludes the proof.
\end{proof}
\begin{proof}[Proof of Theorem \ref{estcon}]
We start noticing that
\begin{eqnarray}\label{conv}
&&\sqrt{n}(\hat\theta_1(\hat k_0)-\hat\theta_1(k_0))\\
&&=\sqrt{n}\left(\frac{1}{\hat k_0}\sum_{i=1}^{\hat k_0}Z_i^2-\frac{1}{ k_0}\sum_{i=1}^{k_0}Z_i^2\right)\notag\\
&& =\mathbf{1}_{\{\hat k_0\leq k_0\}}\left(\sqrt{n} \frac{k_0-\hat k_0}{k_0\hat k_0}\sum_{i=1}^{k_0}\left( Z_i^2-\theta_1\right)-\sqrt{n}\frac{1}{\hat k_0}\sum_{i=\hat k_0}^{k_0}\left( Z_i^2-\theta_1\right)\right)\notag\\
&&\quad+ \mathbf{1}_{\{\hat k_0 >k_0\}}\Biggl(\sqrt{n}
\frac{k_0-\hat k_0}{k_0\hat k_0}\sum_{i=1}^{k_0}\left(
Z_i^2-\theta_1\right)+\sqrt{n}\frac{1}{\hat k_0}\sum_{i=
k_0}^{\hat k_0}\left( Z_i^2-\theta_2\right)\notag\\
&& \quad+\sqrt{n}
\vartheta_n\frac{\hat k_0-k_0}{\hat k_0}\Biggr).\notag
\end{eqnarray}
Since $k_0=[\tau_0 n]$, $\hat k_0=k_0+O_p(\vartheta_n^{-2})$,
and $n\vartheta_n^2 \to \infty$, we have that \eqref{conv}
is $(\sqrt{n}\vartheta_n)^{-1}O_p(1)$, which converges to zero in
probability. Then $\hat\theta_1(\hat k_0)$ and $\hat\theta_1( k_0)$
have the same limiting distribution. Obviously the same result holds for $\hat\theta_2$. Now it is easy to show that the limiting distribution of $\sqrt{n}(\hat\theta_1(k_0),\hat\theta_2(k_0))$ is equal to \eqref{estcon}.

\end{proof}
\begin{proof}[Proof of Theorem \ref{teonon}]

To prove our thesis it is sufficient to show that
\begin{equation}\label{firstp}
\sqrt{n}(\hat V_k-\hat V_{k_0}-(  V_k- V_{k_0}))\stackrel{p}{\to } 0.
\end{equation}
where $V_k$ and $V_{k_0}$ in this proof are the statistics obtained by setting $\sigma(X_t)=1$ into equation \eqref{eq:sde}.
 We rewrite the left-hand side of \eqref{firstp} as follows
 \begin{eqnarray}\label{secondtp}
&& \sqrt{n}\sqrt{\frac{k}{n}\left(1-\frac{k}{n}\right)}\left[\frac{1}{n-k}(\hat S_{n-k}-S_{n-k})-\frac{1}{k}(\hat S_{k}-S_{k})\right]\\
 &&-\sqrt{n}\sqrt{\frac{k_0}{n}\left(1-\frac{k_0}{n}\right)}\left[\frac{1}{n-k_0}(\hat S_{n-k_0}-S_{n-k_0})-\frac{1}{k_0}(\hat S_{k_0}-S_{k_0})\right].\nonumber
 \end{eqnarray}
The first term of \eqref{secondtp} is equal to
\begin{eqnarray}\label{tertp}
&& \sqrt{n}\sqrt{\frac{k}{n}\left(1-\frac{k}{n}\right)}\left[\frac{1}{n-k}(\hat S_{n-k}-S_{n-k})-\frac{1}{k}(\hat S_{k}-S_{k})\right]\nonumber\\
 &&=\sqrt{n}\sqrt{\frac{k}{n}\left(1-\frac{k}{n}\right)}\left[\frac{1}{n-k}\sum_{i=k+1}^n(\hat Z_{i}^2-Z_{i}^2)-\frac{1}{k}\sum_{i=1}^k(\hat Z_{i}^2-Z_{i}^2)\right]\nonumber\\
&&=\sqrt{\frac{n}{n-k}}\sqrt{\frac{k}{n}\left(1-\frac{k}{n}\right)}\frac{1}{\sqrt{n-k}}\sum_{i=k+1}^n(\hat Z_{i}^2-Z_{i}^2)\nonumber\\
&&\quad-\sqrt{\frac{n}{k}}\sqrt{\frac{k}{n}\left(1-\frac{k}{n}\right)}\frac{1}{\sqrt{k}}\sum_{i=1}^k(\hat Z_{i}^2-Z_{i}^2)
 \end{eqnarray}
We observe that
\begin{eqnarray}\label{nonp}
\frac{1}{\sqrt{k}}\sum_{i=1}^k(\hat Z_{i}^2-Z_{i}^2)
&=&\frac{2}{\sqrt{k}}\sum_{i=1}^kZ_i(\hat Z_{i}-Z_{i})+\frac{1}{\sqrt{k}}\sum_{i=1}^k(\hat Z_{i}-Z_{i})^2\nonumber\\
&=&2\mathcal{Z}_1+\mathcal{Z}_2,
\end{eqnarray}
The next step is to prove that \eqref{nonp} tends to zero in probability. In fact,
by simple calculations we can write
\begin{eqnarray*}
\mathcal{Z}_1
&=&\frac{\sqrt{\theta_1}\sqrt{\Delta_n}}{\sqrt{k}}\sum_{i=1}^k(b(X_i)-\hat b(X_i))\frac{W_{i+1}-W_i}{\sqrt{\Delta_n}}\notag\\
&=&\frac{\sqrt{\theta_1}\sqrt{\Delta_n}}{\sqrt{k}}\sum_{i=1}^k\frac{(W_{i+1}-W_i)}{\sqrt{\Delta_n}}\frac{\sum_{j=1}^nK\left(
\frac{X_j-X_i}{h_n}\right)\left(b(X_i)-\frac{
X_{j+1}-X_j}{\Delta_n}\right)}{\sum_{j=1}^nK\left(
\frac{X_j-X_i}{h_n}\right)}
\end{eqnarray*}
Then we have that
\begin{eqnarray*}
E(\mathcal{Z}_1)^2=\frac{\theta_1\Delta_n}{k}E\left(\sum_{i=1}^k\frac{(W_{i+1}-W_i)}{\sqrt{\Delta_n}}\frac{\sum_{j=1}^nK\left(
\frac{X_j-X_i}{h_n}\right)\left(b(X_i)-\frac{
X_{j+1}-X_j}{\Delta_n}\right)}{\sum_{j=1}^nK\left(
\frac{X_j-X_i}{h_n}\right)}\right)^2
\end{eqnarray*}
We note the following fact
\begin{eqnarray}\label{eq:med}
&&E\left(X_{j+1}-X_j-b(X_i)\Delta_n\right)^{2}\notag\\
&&=E\left(\sqrt{\theta_1}(W_{j+1}-W_j)+\int_{t_{j}}^{t_{j+1}}[b(X_s)-b(X_i)]ds\right)^{2}\notag\\
&&\leq 2\left\{\theta_1E(W_{j+1}-W_j)^{2}
+E\left(\int_{t_{j}}^{t_{j+1}}[b(X_s)-b(X_i)]ds\right)^{2}\right\}\notag\\
&&=2\left\{\theta_1\Delta_n
+E\left(\int_{t_{j}}^{t_{j+1}}[b(X_s)-b(X_i)]ds\right)^{2}\right\}.
\end{eqnarray}
Furthermore
\begin{equation}\label{eq:med2}
E\left(\int_{t_{j}}^{t_{j+1}}[b(X_s)-b(X_i)]ds\right)^{2}=O(\Delta_n^{2})
\end{equation}
because the drift coefficient $b(\cdot)$ has continuous derivatives, therefore is locally Lipschitz (see e.g. A\"it-Sahalia, 1996), and the process $X_t$ does not explode in a finite time. Then
by means of relationships \eqref{eq:med} and \eqref{eq:med2}, we can claim that 
\begin{equation}\label{eq:bound}
\frac{X_{j+1}-X_j-b(X_i)\Delta_n}{\Delta_n}=O_p(1).
\end{equation}
The expression \eqref{eq:bound} allows us to write
\begin{eqnarray*}
E(\mathcal{Z}_1)^2\leqslant C\frac{\theta_1\Delta_n}{k} E\left(\sum_{i=1}^k\frac{(W_{i+1}-W_i)}{\sqrt{\Delta_n}}\right)^2=C\theta_1 \Delta_n\to 0
\end{eqnarray*}
for some real constant $C$. The same arguments permit us to obtain that
\begin{equation*} 
\mathcal{Z}_2=\frac{\Delta_n}{\sqrt{k}}\sum_{i=1}^k(b(X_i)-\hat
b(X_i))^2\stackrel{p}{\to}0.
\end{equation*}

Similarly we can develop the second term of \eqref{secondtp}, so the proof is complete.

\end{proof}

\end{document}